\documentclass[12pt,reqno]{article}
\usepackage{amsfonts}
\usepackage{amsmath, amsthm, amsfonts, amssymb, color}
\usepackage{mathrsfs}
\usepackage{fixmath}
\usepackage{graphicx}
\allowdisplaybreaks

\allowdisplaybreaks

\newtheorem{thm}{Theorem}[section]

\newtheorem{lem}[thm]{Lemma}

\newtheorem{ass}[thm]{Assumption}
\newtheorem{rem}{Remark}

\theoremstyle{definition}
\newtheorem{defn}{Definition}[section]

\numberwithin{equation}{section}
\author{
	Zhuoqi Liu,\thanks{Department of Mathematics, Shanghai Normal University, Shanghai 200234, China}
	\and 
	Qian Guo,\thanks{Department of Mathematics, Shanghai Normal University, Shanghai 200234, China}
	\and 
	 Shuaibin Gao,\thanks{Corresponding author. Department of Mathematics, Shanghai Normal University, Shanghai 200234, China, shuaibingao@163.com}
	}
\begin{document}

	%\begin{frontmatter}
	
	
	
	\title{Mean-square convergence and stability of the backward Euler method for  stochastic differential delay equations with highly nonlinear growing coefficients}
	
	%\author{
		%{\bf  Shuaibin Gao$^a$, \,Weifeng Wang$^{a}$, \, Junhao Hu$^{a*}$}\\
		%\footnotesize{$^{a}$ College of Mathematics and statistics, South-Central University For Nationalities,
			%Wuhan 430074, China}\\
		%\footnotesize{wang youxiang, shuaibingao@163.com, junhaohu74@163.com, }}
	
	%%\author[label1]{Shuaibin Gao}
	%%\author[label1]{Junhao Hu}
	%%\author[label2]{Chenggui Yuan\corref{cor1}}\ead{chengguiyuan@hotmail.com}
	
	%%\cortext[cor1]{Corresponding author.}
	%%\address[label1]{College of Mathematics and statistics, South-Central University For Nationalities,
		%%Wuhan 430074, China}
	%%\address[label2]{Department of Mathematics, Swansea University, UK }
	\maketitle

		\begin{abstract}
			
			Over the last few decades, the numerical methods for stochastic differential delay equations (SDDEs) have been investigated and developed by many scholars. Nevertheless, there is still little work to be completed. By virtue of the novel technique, this paper focuses on the mean-square convergence and stability of the backward Euler method (BEM) for SDDEs whose drift and diffusion coefficients can both grow polynomially.
			The upper mean-square error bounds of BEM are obtained. Then the convergence rate, which is one-half, is revealed without using the moment boundedness of numerical solutions. Furthermore, under fairly general conditions, the novel technique is applied to prove that the BEM can inherit the exponential mean-square stability with a simple proof. At last, two numerical experiments are implemented to illustrate the reliability of the theories.
		\end{abstract}
		

	\section{Introduction}
	Stochastic differential equations (SDEs) have been investigated by many scholars due to their extensive applications in many fields, including control problems, finance, biology, population model, communication, etc \cite{1,14,18}. However, it is difficult to obtain the exact solutions to SDEs. So using numerical algorithms to aquire approximations is a meaningful way to analyze the properties of solutions. It is well known that  Euler-Maruyama (EM) method is one of the most popular numerical algorithms for SDEs \cite{14,18}. 
	Unfortunately, the divergence of EM scheme for SDEs with super-linear coefficients was proved in \cite{11}. 
	Whereafter, different kinds of modified EM methods have been established to approximate nonlinear SDEs, such as truncated EM method\cite{15,20}, tamed EM method\cite{12,25}, stopped EM method\cite{17}, multilevel EM method\cite{2}, projected EM method\cite{6} and others. Furthermore, the implicit methods have also been studied and developed on account of their better convergence rates in the last decades \cite{3,4,14,26,30}.
	
	The scholars use stochastic differential delay equations(SDDEs) to describe a class of more applicative systems which not only depend on the present state but also depend on the past state\cite{7,18}. 
	Like SDEs, the numerical methods of SDDEs have also been widely discussed. 
	The modified EM methods for SDDEs were analyzed in\cite{fei2020,8,9,16,song2022}, while the implicit EM methods were investigated in \cite{201,202,203,130,129}. 
	It is worth noting that the diffusion coefficients of the equations in \cite{201,202,203,130,129} can not grow super-linearly, which has a adverse effect on the development of the implicit methods.
	In order to eliminate this adverse effect, the papers \cite{28,34,29} and \cite{27,131} focused on studying the backward Euler method (BEM) and split-step method for SDDEs respectively, whose drift and diffusion coefficients can grow super-linearly.
	However, the convergence rate was not shown in \cite{28}; 
	%	the present state variable in the diffusion coefficient cannot grow super-linearly in \cite{27};
	there is no constant in the diffusion coefficient in \cite{34,29}, which is a strong constraint.
	Therefore, by exploiting the novel technique, the first goal of our paper is to investigate the strong convergence rate of the BEM for SDDEs with highly nonlinear drift and diffusion coefficients under the weaker conditions.

	As is known to all, in addition to the convergence, the long-time stability of the numerical solution is also worth studying.
	The stabilities of implicit EM methods for SDEs were given in \cite{103,104,21,40}.
	In the rest of this paragraph, we only discuss the stability of the implicit scheme for SDDEs.
	%	The stabilities of the BEM methods for SDDEs with linear diffusion coefficients were analyzed in \cite{}. 
	When the diffusion coefficients satisfy $|g(x(t),x(t-\tau))|^2\leq c_1|x(t)|^2+c_2|x(t-\tau)|^2$ (where $c_1,c_2$ are positive constants and $x(t-\tau)$ is the delay term), the  stabilities of the BEM solutions to SDDEs were studied in \cite{121,120,119,111,110,117,116,112,118,109}. 
	The theories in \cite{34,28,29} can not cover the equations with diffusion coefficient $g(x(t),x(t-\tau))=x(t-\tau)$. Similarly, to a degree, Assumption 2.2 in \cite{123} is also a bit strong. Moreover, the stabilities of split-step methods were analyzed in \cite{126,127,128,129,130}.
	Especially, it should be noted that  the constraint of coefficients was relaxed in \cite{122}, but the  locally Lipschitz conditions of  coefficients were used in the proof process. 
	Hence, the second goal of our paper is to prove that the numerical solutions to BEM for SDDEs are exponentially mean-square stable without using the  locally Lipschitz conditions.
	That is, by making use of the novel technique, 
	the mean-square stability of the numerical solutions can be obtained easily under the weaker conditions .
	
	Let's summarize the main contributions of our paper. Under the fairly general conditions, by borrowing the techniques from\cite{3,31,30}, we investigate the strong convergence rate and exponential mean-square stability of the BEM for SDDEs whose drift and diffusion coefficients can grow super-linearly.

	To use the novel technique, we introduce the crucial equality
	\begin{equation}\label{cru}
		2\langle x-y,x \rangle=|x|^2-|y|^2+|x-y|^2,~~~~~~~\forall x, y\in\mathbb{R}^d,
	\end{equation}
	which will play an important role in our paper.		
	Moreover, the discussions of the comparison between explicit and implicit numerical schemes can be found in \cite{32,33,31,30} and references therein.
	
	This paper is organized as follows. In Section 2, we introduce some necessary notations and prove that the global error in mean-square sence is controlled by the local error. Section 3 gives the convergence rate of BEM without using the moment boundedness of numerical solutions. In Section 4, under a stronger condition, the convergence rate is given by a much simpler proof. In Section 5, we present the exponential mean-square stability of BEM. In Section 6, two numerical examples are considered to illustrate the reliability of the theories.

	\section{Error bounds for BEM}
	 
	Let $|\cdot |$ and$<\cdot,\cdot>$ denote the Euclidean norm and the inner product of vectors in $\mathbb{R}^d$. 
	We use $\mathbb{N}$ to denote the set of all positive integers.
	If $A\in \mathbb{R}^{d\times m}$ is a matrix, its trace norm is denoted by$\|A\|=\sqrt{trace(A^TA)}$, where $A^T$ is the transpose of matrix A, $d,m\in\mathbb{N}$.
	Let $\big ( \Omega, \mathcal{F}, \{\mathcal{F}_t\}_{t\in[0 , T]}, \mathbb{P} \big )$ stand for a complete probability space with a filtration  $\{\mathcal{F}_t\}_{t\in[0 , T]}$ satisfying the usual conditions (i.e., it is increasing and right continuous while $ \mathcal{F}_0$ contains all $\mathbb{P}$-null sets).
	Let $\mathbb{E}$ be the probability expectation with respect to $\mathbb{P}$.
	Let $\mathcal{L}^r =\mathcal{L}^r( \Omega, \mathbb{R}^d )$ be the family of $\mathbb{R}^d$-valued random variables $\xi$ satisfying $\mathbb{E}|\xi|^r<\infty$, $r\in \mathbb{N}$.
	Let $\mathcal{C} ([-\tau,0];\mathbb{R}^d)$ stand for the family of all continuous functions from  $[-\tau,0]$  to   $\mathbb{R}^d$ with the norm $|\varphi|_\infty=\sup_{- \tau \leq \theta \leq 0}|\varphi ( \theta )|$.  
	%Let $\mathcal{C}_{\mathcal{F}_0}^b ([-\tau,0];\mathbb{R}^d)$ be the family of all bounded, $\mathcal{F}_0$-measurable, $\mathcal{C}([-\tau,0];\mathbb{R}^d)$-valued functions. 
	And denote by C a generic positive constant which is independent of time stepsize.
	
	Now we consider the nonlinear SDDE of the form
	\begin{equation}\label{sdde1}
		d x(t)=\alpha\left(x(t), x(t-\tau)\right) d t+\beta\left(x(t), x(t-\tau)\right) dW(t),
	\end{equation}
	on $t\geq 0$ with the initial data
	\begin{equation}\label{initial}
		x_{0}=\varphi=\{ \varphi ( \theta ) : - \tau \leq \theta \leq 0 \}\in \mathcal{C} ([ - \tau , 0 ];\mathbb{R}^d),
	\end{equation}
	where
	$\alpha :\mathbb{R}^d\times \mathbb{R}^d \rightarrow \mathbb{R}^d$ and  $\beta :\mathbb{R}^d\times \mathbb{R}^d \rightarrow \mathbb{R}^{d\times m}$.
	Moreover, $W(t)$ is an m-dimensional Brownian motion.
	
	Now we construct the BEM for SDDEs. Suppose that there exist two positive integers $N,M$ such that $\Delta=\frac{T}{N}=\frac{\tau}{M}$, where $\Delta$ is the step size. Define
	\begin{align}\label{bem1}
		\left \{\begin{array}{l}
			Z_{n}=\varphi (n\Delta),~~~~~~~~~~~n=-M , -M+1 , ... , 0, \\
			Z_{n}=Z_{n-1}+\alpha(Z_{n} , Z_{n-M}) \Delta+\beta(Z_{n-1} , Z_{n-1-M}) \Delta W_{n-1},~~n=1 , 2, ... , N,
		\end{array}\right .
	\end{align}
	where $\Delta W_{n-1}:=W(t_n)-W(t_{n-1})$.

	Before analyzing the errors between exact solutions and numerical solutions, some necessary assumptions should be imposed.

	\begin{ass}\label{a1}
		There exist constants $K>0$ and $q>2$ such that
		\begin{equation*}
			\begin{split}
				&\langle x-\bar{x} , \alpha(x,y)-\alpha(\bar{x},\bar{y})\rangle +\frac{(q-1)}{2}\|\beta(x,y)-\beta(\bar{x},\bar{y})\|^{2}\\
				\leq& K(|x-\bar{x}|^{2}+|y-\bar{y}|^{2}),
			\end{split}
		\end{equation*}
		for any $x,y,\bar{x},\bar{y}\in \mathbb{R}^d$.
	\end{ass}
By Assumption \ref{a1}, we can get that
	\begin{equation}\label{*}
			\langle x-\bar{x} , \alpha(x,y)-\alpha(\bar{x},y)\rangle 
			\leq K|x-\bar{x}|^{2},
	\end{equation}	
		for any $x,y,\bar{x}\in \mathbb{R}^d$.
	When $K\Delta<1$, (\ref{bem1}) has a unique solution due to the fixed point theorem. Then the BEM is well defined \cite{40}.
	\begin{ass}\label{a2}
		Suppose that the SDDE (\ref{sdde1}) admits a unique $\{\mathcal{F}_t\}_{t\in[0 , T]}$-adapted solution with continuous sample paths, 
		and $\sup_{s\in[0 , T]}\mathbb{E}|x(s)|^2<\infty$, $\sup_{s\in[0 , T]}\mathbb{E}|\alpha(x(s),x(s-\tau))|^2<\infty$ hold. Furthermore, assume that the BEM admits a unique $\{\mathcal{F}_{t_n}\}_{n\in\{1,2,\cdots,N\}}$-adapted solution$\{Z_n\}_{n=0}^N$ as well.

	\end{ass}
	
	\begin{lem}\label{le1}
		Let Assumption \ref{a2} hold. For ${n\in\{1,2,...,N\}}$, $N\in\mathbb{N}$, we have

		\begin{equation}\label{equ1}			
			\begin{split}
				&\big|[x(t_n)-Z_n]-\Delta[\alpha(x(t_n),x(t_{n-M}))-\alpha(Z_n,Z_{n-M})]\big|^2\\
				=&\big|[x(t_{n-1})-Z_{n-1}]-\Delta[\alpha(x(t_{n-1}),x(t_{n-1-M}))-\alpha(Z_{n-1},Z_{n-1-M})]\big|^2\\
				&-\Delta^2\big|\alpha(x(t_{n-1}),x(t_{n-1-M}))-\alpha(Z_{n-1},Z_{n-1-M})\big|^2\\
				&+\big|[\beta(x(t_{n-1}),x(t_{n-1-M}))-\beta(Z_{n-1},Z_{n-1-M})]\Delta W_{n-1}\big|^2+|\mathcal{R}_n|^2\\	
				&+2\Delta \langle x(t_{n-1})-Z_{n-1} , \alpha(x(t_{n-1}),x(t_{n-1-M}))-\alpha(Z_{n-1},Z_{n-1-M})  \rangle\\
				&+2 \langle [x(t_{n-1})-Z_{n-1}]-\Delta[\alpha(x(t_{n-1}),x(t_{n-1-M}))-\alpha(Z_{n-1},Z_{n-1-M})] ,\\
				&~~~~~~[\beta(x(t_{n-1}),x(t_{n-1-M}))-\beta(Z_{n-1},Z_{n-1-M})]\Delta W_{n-1}  \rangle\\
				&+2\langle [x(t_{n-1})-Z_{n-1}]-\Delta[\alpha(x(t_{n-1}),x(t_{n-1-M}))-\alpha(Z_{n-1},Z_{n-1-M})] , \mathcal{R}_n \rangle\\
				&+2\Delta \langle \alpha(x(t_{n-1}),x(t_{n-1-M}))-\alpha(Z_{n-1},Z_{n-1-M}),\\ &~~~~~~~~[\beta(x(t_{n-1}),x(t_{n-1-M}))-\beta(Z_{n-1},Z_{n-1-M})]\Delta W_{n-1}  \rangle\\
				&+2\Delta \langle \alpha(x(t_{n-1}),x(t_{n-1-M}))-\alpha(Z_{n-1},Z_{n-1-M}) ,\mathcal{R}_n  \rangle\\
				&+2\langle[\beta(x(t_{n-1}),x(t_{n-1-M}))-\beta(Z_{n-1},Z_{n-1-M})]\Delta W_{n-1} , \mathcal{R}_n\rangle,
			\end{split}
		\end{equation}
		
		where
		
		\begin{equation}\label{Rn}
			\begin{split}
				\mathcal{R}_n=&\int_{t_{n-1}}^{t_n} [\alpha(x(s) , x(s-\tau))-\alpha(x(t_n)-x(t_{n-M}))] ds\\
				&+\int_{t_{n-1}}^{t_n} [\beta(x(s) , x(s-\tau))-\beta(x(t_{n-1})-x(t_{n-1-M}))]dW(s).
			\end{split}
		\end{equation}
	\end{lem}
	
	\begin{proof} 
		
		From the definitions of (\ref{sdde1}) and (\ref{bem1}), one can see that

		\begin{equation*}
			\begin{split}
				&[x(t_n)-Z_n]-\Delta[\alpha(x(t_n),x(t_{n-M}))-\alpha(Z_n,Z_{n-M})]\\
				=&[x(t_{n-1})-Z_{n-1}]-\Delta[\alpha(x(t_{n-1}),x(t_{n-1-M}))-\alpha(Z_{n-1},Z_{n-1-M})]\\
				&+\Delta[\alpha(x(t_{n-1}),x(t_{n-1-M}))-\alpha(Z_{n-1},Z_{n-1-M})]\\
				&+[\beta(x(t_{n-1}),x(t_{n-1-M}))-\beta(Z_{n-1},Z_{n-1-M})]\Delta W_{n-1}\\
				&+\mathcal{R}_n ,
			\end{split}
		\end{equation*}
		where $\mathcal{R}_n$ is defined by (\ref{Rn}).
		Then we have
		\begin{equation*}			
			\begin{split}
				&\big|[x(t_n)-Z_n]-\Delta[\alpha(x(t_n),x(t_{n-M}))-\alpha(Z_n,Z_{n-M})]\big|^2\\
				=&\big|[x(t_{n-1})-Z_{n-1}]-\Delta[\alpha(x(t_{n-1}),x(t_{n-1-M}))-\alpha(Z_{n-1},Z_{n-1-M})]\big|^2\\
				&+\Delta^2\big|\alpha(x(t_{n-1}),x(t_{n-1-M}))-\alpha(Z_{n-1},Z_{n-1-M})\big|^2\\
				&+\big|[\beta(x(t_{n-1}),x(t_{n-1-M}))-\beta(Z_{n-1},Z_{n-1-M})]\Delta W_{n-1}\big|^2+|\mathcal{R}_n|^2\\	
				&+2\Delta \langle x(t_{n-1})-Z_{n-1} , \alpha(x(t_{n-1}),x(t_{n-1-M}))-\alpha(Z_{n-1},Z_{n-1-M}) \rangle\\
				&-2\Delta^2\big|\alpha(x(t_{n-1}),x(t_{n-1-M}))-\alpha(Z_{n-1},Z_{n-1-M})\big|^2\\
				&+2 \langle [x(t_{n-1})-Z_{n-1}]-\Delta[\alpha(x(t_{n-1}),x(t_{n-1-M}))-\alpha(Z_{n-1},Z_{n-1-M})] ,\\
				&~~~~~~[\beta(x(t_{n-1}),x(t_{n-1-M}))-\beta(Z_{n-1},Z_{n-1-M})]\Delta W_{n-1}  \rangle\\
				&+2\langle[x(t_{n-1})-Z_{n-1}]-\Delta[\alpha(x(t_{n-1}),x(t_{n-1-M}))-\alpha(Z_{n-1},Z_{n-1-M})] , \mathcal{R}_n\rangle\\
				&+2\Delta \langle \alpha(x(t_{n-1}),x(t_{n-1-M}))-\alpha(Z_{n-1},Z_{n-1-M}),\\ &~~~~~~~~~[\beta(x(t_{n-1}),x(t_{n-1-M}))-\beta(Z_{n-1},Z_{n-1-M})]\Delta W_{n-1}  \rangle\\
				&+2\Delta \langle \alpha(x(t_{n-1}),x(t_{n-1-M}))-\alpha(Z_{n-1},Z_{n-1-M}) ,\mathcal{R}_n  \rangle\\
				&+2\langle[\beta(x(t_{n-1}),x(t_{n-1-M}))-\beta(Z_{n-1},Z_{n-1-M})]\Delta W_{n-1} , \mathcal{R}_n\rangle.
			\end{split}
		\end{equation*}
		Rearranging the above equation gives the result.
	\end{proof}
	\begin{lem}\label{le2}
		Let Assumptions \ref{a1},\ref{a2} hold and $0<2K\Delta<1$. Then for all $n\in \{1,2,...,N\}$, $N\in\mathbb{N}$, we derive that
		
		\begin{equation*}	
			\mathbb{E}\big| \mathcal{R}_n\big|^2<\infty,  ~~~~~ ~~~  \mathbb{E}\big| x(t_n)-Z_n\big|^2<\infty,
		\end{equation*}	
		\begin{equation*}	
			\mathbb{E}\big| [x(t_n)-Z_n]-\Delta[\alpha(x(t_n),x(t_{n-M}))-\alpha(Z_n,Z_{n-M})]\big
			|^2<\infty,
		\end{equation*}	
		\begin{equation*}	
			\mathbb{E}\big|\alpha(x(t_n),x(t_{n-M}))-\alpha(Z_n,Z_{n-M})\big|^2<\infty,
		\end{equation*}	
		\begin{equation*}	
			\mathbb{E}\|\beta(x(t_n),x(t_{n-M}))-\beta(Z_n,Z_{n-M})]\Delta W_{n-1} \|^2<\infty.
		\end{equation*}	
	\end{lem}
	\begin{proof}
		For all $n\in \{1,2,...,N\}$, $N\in\mathbb{N}$ , we prove $ \mathbb{E}| \mathcal{R}_n|^2<\infty$ firstly. We have already known that $\sup_{s\in[0 , T]}\mathbb{E}|x(s)|^2<\infty$  and 
		$\sup_{s\in[0 , T]}\mathbb{E}|\alpha(x(s),x(s-\tau))|^2<\infty$. By Assumption \ref{a1}, for $s\in[t_{n-1},t_n]$ we can get
		\begin{equation*}	
			\begin{split}	
				& \mathbb{E}\big\|\beta(x(s) , x(s-\tau))-\beta(x(t_{n-1})-x(t_{n-1-M}))\big\|^2\\
				\leq &\frac{2K}{(q-1)}\mathbb{E}\big|x(s)-x(t_{n-1})\big|^2+\frac{2K}{(q-1)}\mathbb{E}\big|x(s-\tau)-x(t_{n-1-M})\big|^2\\
				&-\frac{2}{(q-1)}\mathbb{E}\big\langle x(s)-x(t_{n-1}),\alpha(x(s),x(s-\tau))-\alpha(x(t_{n-1}),x(t_{n-1-M}))\big\rangle\\
				\leq &\frac{2K+1}{(q-1)}\mathbb{E}\big| x(s)-x(t_{n-1})\big|^2+\frac{2K}{(q-1)}\mathbb{E}\big|x(s-\tau)-x(t_{n-1-M})\big|^2\\
				&+\frac{1}{(q-1)} \mathbb{E}\big|\alpha(x(s),x(s-\tau))-\alpha(x(t_{n-1}),x(t_{n-1-M}))\big|^2
				<\infty.
			\end{split}	
		\end{equation*}	
		Relying on the obtained information above, for ${n\in\{1,2,...,N\}}$, $N\in\mathbb{N}$, $ \mathbb{E}| \mathcal{R}_n|^2<\infty $  can be proved due to the It\^{o} isometry. Taking expectations on both sides of (\ref{equ1}), it is easy to get\\
		\begin{equation}\label{equ2}
			\begin{split}
				& \mathbb{E}\big|[x(t_n)-Z_n]-\Delta[\alpha(x(t_n),x(t_{n-M}))-\alpha(Z_n,Z_{n-M})]\big|^2\\
				%=& \mathbb{E}[|[x(t_{n-1})-Z_{n-1}]-[\alpha(x(t_{n-1}),x(t_{n-1-M}))-\alpha(Z_{n-1},Z_{n-1-M})]\Delta|^2]\\
				%&- \Delta^2\mathbb{E}[|[\alpha(x(t_{n-1}),x(t_{n-1-M}))-\alpha(Z_{n-1},Z_{n-1-M})]|^2]\\
				%&+ \mathbb{E}[\|[\beta(x(t_{n-1}),x(t_{n-1-M}))-\beta(Z_{n-1},Z_{n-1-M})]\Delta W_{n-1}\|^2+ \mathbb{E}[|\mathcal{R}_n|^2]]\\	
				%&+2 \Delta\mathbb{E}[ \langle [x(t_{n-1})-Z_{n-1} ], [\alpha(x(t_{n-1}),x(t_{n-1-M}))-\alpha(Z_{n-1},Z_{n-1-M})]  \rangle]\\
				%&+2 \mathbb{E}[\langle [x(t_{n-1})-Z_{n-1}]-\Delta[\alpha(x(t_{n-1}),x(t_{n-1-M}))-\alpha(Z_{n-1},Z_{n-1-M})] ,\\
				%&[\beta(x(t_{n-1}),x(t_{n-1-M}))-\beta(Z_{n-1},Z_{n-1-M})]\Delta W_{n-1}  \rangle]\\
				%&+2\mathbb{E}[\langle[x(t_{n-1})-Z_{n-1}]-\Delta[\alpha(x(t_{n-1}),x(t_{n-1-M}))-\alpha(Z_{n-1},Z_{n-1-M})] , \mathcal{R}_n\rangle]\\
				%&+2\Delta \mathbb{E}[\langle [\alpha(x(t_{n-1}),x(t_{n-1-M}))-\alpha(Z_{n-1},Z_{n-1-M})] ,\\ &[\beta(x(t_{n-1}),x(t_{n-1-M}))-\beta(Z_{n-1},Z_{n-1-M})]\Delta W_{n-1}  \rangle]\\
				%&+2\Delta \mathbb{E}[\langle [\alpha(x(t_{n-1}),x(t_{n-1-M}))-\alpha(Z_{n-1},Z_{n-1-M})] ,\mathcal{R}_n  \rangle]\\
				%&+2\mathbb{E}[\langle[\beta(x(t_{n-1}),x(t_{n-1-M}))-\beta(Z_{n-1},Z_{n-1-M})]\Delta W_{n-1} , \mathcal{R}_n\rangle]\\
				\leq& 2\mathbb{E}\big|[x(t_{n-1})-Z_{n-1}]-\Delta[\alpha(x(t_{n-1}),x(t_{n-1-M}))-\alpha(Z_{n-1},Z_{n-1-M})]\big|^2\\
				&+2 \mathbb{E}\big\|[\beta(x(t_{n-1}),x(t_{n-1-M}))-\beta(Z_{n-1},Z_{n-1-M})]\Delta W_{n-1}\big\|^2+ 4\mathbb{E}|\mathcal{R}_n|^2\\	
				&+2 \Delta\mathbb{E}\big\langle x(t_{n-1})-Z_{n-1}, \alpha(x(t_{n-1}),x(t_{n-1-M}))-\alpha(Z_{n-1},Z_{n-1-M})\big \rangle\\
				&+2 \mathbb{E}\big\langle [x(t_{n-1})-Z_{n-1}]-\Delta[\alpha(x(t_{n-1}),x(t_{n-1-M}))-\alpha(Z_{n-1},Z_{n-1-M})] ,\\
				&~~~~~~~~~~[\beta(x(t_{n-1}),x(t_{n-1-M}))-\beta(Z_{n-1},Z_{n-1-M})]\Delta W_{n-1}  \big \rangle\\
				&+2\Delta \mathbb{E}\big\langle \alpha(x(t_{n-1}),x(t_{n-1-M}))-\alpha(Z_{n-1},Z_{n-1-M}),\\
				&~~~~~~~~~~~[\beta(x(t_{n-1}),x(t_{n-1-M}))-\beta(Z_{n-1},Z_{n-1-M})]\Delta W_{n-1} \big \rangle.\\
			\end{split}
		\end{equation}	
		Next, we use the inductive reasoning to prove\\   $\mathbb{E}\big| [x(t_n)-Z_n]-\Delta[\alpha(x(t_n),x(t_{n-M}))-\alpha(Z_n,Z_{n-M})]\big|^2<\infty$. One can observe that
		\begin{equation*}	
			\begin{split}
				& [x(0)-Z_0]-\Delta[\alpha(x(0),x(-\tau))-\alpha(Z_0,Z_{-M})]\\
				=&[\alpha(x(0),x(-\tau))-\alpha(Z_0,Z_{-M})]\\
				=&\alpha(x(0),x(-\tau))-\alpha(Z_0,Z_{-M})\\
				=&[\beta(x(0),x(-\tau))-\beta(Z_0,Z_{-M})]\Delta W_0\\
				=&0.
			\end{split}
		\end{equation*}	
		According to (\ref{equ2}), we konw that $\mathbb{E}\big|[x(t_1)-Z_1]-\Delta[\alpha(x(t_1),x(t_{1-M}))-\alpha(Z_1,Z_{1-M})]\big|^2\leq4\mathbb{E}|\mathcal{R}_1|^2<\infty$. Now we suppose that\\
		 $\mathbb{E}\big|[x(t_{n-1})-Z_{n-1}] -\Delta[ \alpha(x(t_{n-1}),x(t_{n-1-M})) -\alpha(Z_{n-1},Z_{n-1-M})]\big|^2<\infty$, then we get from Assumption \ref{a1} that
		\begin{equation*}	
			\begin{split}
				&\mathbb{E}\big|[x(t_{n-1})-Z_{n-1}]-\Delta[\alpha(x(t_{n-1}),x(t_{n-1-M}))-\alpha(Z_{n-1},Z_{n-1-M})]\big|^2\\
				=&\mathbb{E}\big|x(t_{n-1})-Z_{n-1}\big|^2\\
				&-2 \Delta\mathbb{E}\big \langle x(t_{n-1})-Z_{n-1} , \alpha(x(t_{n-1}),x(t_{n-1-M}))-\alpha(Z_{n-1},Z_{n-1-M})  \big \rangle\\
				&+\Delta^2\mathbb{E}\big|\alpha(x(t_{n-1}),x(t_{n-1-M}))-\alpha(Z_{n-1},Z_{n-1-M})\big|^2\\
				\geq&\mathbb{E}\big|x(t_{n-1})-Z_{n-1}\big|^2\\
				&-2 K\Delta \mathbb{E}\big|x(t_{n-1})-Z_{n-1}\big|^2-2 K\Delta \mathbb{E}\big|x(t_{n-1-M})-Z_{n-1-M}\big|^2\\
				&+\Delta^2\mathbb{E}\big|\alpha(x(t_{n-1}),x(t_{n-1-M}))-\alpha(Z_{n-1},Z_{n-1-M})\big|^2.
			\end{split}
		\end{equation*}	
		%	\begin{equation}	
			%	\begin{split}
				%&\mathbb{E}[|[x(t_{n-1})-Z_{n-1}]-[\alpha(x(t_{n-1}),x(t_{n-1-M}))-\alpha(Z_{n-1},Z_{n-1-M})]|^2]\\
				%\geq&\mathbb{E}[|x(t_{n-1})-Z_{n-1}|^2]\\
				%&-2 K\Delta \mathbb{E}[|x(t_{n-1})-Z_{n-1}|^2]-2 K\Delta \mathbb{E}[|x(t_{n-1-M})-Z_{n-1-M}|^2]\\
				%&+\Delta^2\mathbb{E}[|\alpha(x(t_{n-1}),x(t_{n-1-M}))-\alpha(Z_{n-1},Z_{n-1-M})|^2]
				%	\end{split}
			%	\end{equation}	
		Thus,
		\begin{equation*}	
			\begin{split}
				&\mathbb{E}\big|[x(t_{n-1})-Z_{n-1}]-\Delta[\alpha(x(t_{n-1}),x(t_{n-1-M}))-\alpha(Z_{n-1},Z_{n-1-M})]\big|^2\\
				&+2 K\Delta \mathbb{E}\big|x(t_{n-1-M})-Z_{n-1-M}\big|^2\\
				\geq&(1-2K\Delta)\mathbb{E}\big|x(t_{n-1})-Z_{n-1}\big|^2\\
				&+\Delta^2\mathbb{E}\big|\alpha(x(t_{n-1}),x(t_{n-1-M}))-\alpha(Z_{n-1},Z_{n-1-M})\big|^2.
			\end{split}
		\end{equation*}	
		Based on the assumptions above, we have
		\begin{equation*}
				\begin{split}	
			\infty>&(1-2K\Delta)\mathbb{E}\big|x(t_{n-1})-Z_{n-1}\big|^2\\
			&+\Delta^2\mathbb{E}\big|\alpha(x(t_{n-1}),x(t_{n-1-M}))-\alpha(Z_{n-1},Z_{n-1-M})\big|^2.
				\end{split}
		\end{equation*}	
		Therefore, for $0\leq2K\varDelta<1$, one can derive that 
		\begin{equation*}	
			\begin{split}
				&\mathbb{E}\big|x(t_{n-1})-Z_{n-1}\big|^2<\infty, \\
				&\mathbb{E}\big|\alpha(x(t_{n-1}),x(t_{n-1-M}))-\alpha(Z_{n-1},Z_{n-1-M})\big|^2<\infty,
			\end{split}
		\end{equation*}
		and
		\begin{equation*}	
			\begin{split}
				& \mathbb{E}\big\|\beta(x(t_{n-1}) , x(t_{n-1-M}))-\beta(Z_{n-1},Z_{n-1-M})\big\|^2\\
				\leq &\frac{2K+1}{(q-1)}\mathbb{E}\big| x(t_{n-1})-Z_{n-1}\big|^2+\frac{2K}{(q-1)}\mathbb{E}\big| x(t_{n-1-M})-Z_{n-1-M}\big|^2\\
				&+\frac{1}{(q-1)} \mathbb{E}\big|\alpha(x(t_{n-1}) , x(t_{n-1-M}))-\alpha(Z_{n-1},Z_{n-1-M})\big|^2
				<\infty.
			\end{split}
		\end{equation*}	
		Moreover, using the properties of Brownian motion and It\^{o}'s isometry gives that 
		\begin{equation*}	
			\begin{split}
				&\mathbb{E}\big|[\beta(x(t_{n-1}),x(t_{n-1-M}))-\beta(Z_{n-1},Z_{n-1-M})]\Delta W_{n-1}\big|^2\\
				=&\Delta\mathbb{E}\big\|\beta(x(t_{n-1}),x(t_{n-1-M}))-\beta(Z_{n-1},Z_{n-1-M})\big\|^2,\\
				& \mathbb{E}\big\langle [x(t_{n-1})-Z_{n-1}]-\Delta[\alpha(x(t_{n-1}),x(t_{n-1-M}))-\alpha(Z_{n-1},Z_{n-1-M})] ,\\
				&~~~~[\beta(x(t_{n-1}),x(t_{n-1-M}))-\beta(Z_{n-1},Z_{n-1-M})]\Delta W_{n-1}  \big \rangle=0,\\
				&\mathbb{E}\big\langle \alpha(x(t_{n-1}),x(t_{n-1-M}))-\alpha(Z_{n-1},Z_{n-1-M}),\\ &~~~~[\beta(x(t_{n-1}),x(t_{n-1-M}))-\beta(Z_{n-1},Z_{n-1-M})]\Delta W_{n-1} \big \rangle=0.
				\end{split}
		\end{equation*}	
		Using Assumption \ref{a1}, we get from (\ref{equ2}) that

		\begin{equation*}	
			\begin{split}
				& \mathbb{E}\big|[x(t_n)-Z_n]-\Delta[\alpha(x(t_n),x(t_{n-M}))-\alpha(Z_n,Z_{n-M})]\big|^2]\\
				\leq& 2\mathbb{E}\big|[x(t_{n-1})-Z_{n-1}]-\Delta[\alpha(x(t_{n-1}),x(t_{n-1-M}))-\alpha(Z_{n-1},Z_{n-1-M})]\big|^2\\
				&+2 \Delta \mathbb{E}\big\|\beta(x(t_{n-1}),x(t_{n-1-M}))-\beta(Z_{n-1},Z_{n-1-M})\big\|^2+ 4\mathbb{E}|\mathcal{R}_n|^2\\
				&+2 \Delta\mathbb{E}\big\langle x(t_{n-1})-Z_{n-1} , \alpha(x(t_{n-1}),x(t_{n-1-M}))-\alpha(Z_{n-1},Z_{n-1-M}) \big \rangle\\
				\leq& 2\mathbb{E}\big|[x(t_{n-1})-Z_{n-1}]-\Delta[\alpha(x(t_{n-1}),x(t_{n-1-M}))-\alpha(Z_{n-1},Z_{n-1-M})]\big|^2\\
				&+2 \Delta \mathbb{E}\big\|\beta(x(t_{n-1}),x(t_{n-1-M}))-\beta(Z_{n-1},Z_{n-1-M})\big\|^2+ 4\mathbb{E}|\mathcal{R}_n|^2\\	
				&+2K\Delta\mathbb{E}\big|x(t_{n-1})-Z_{n-1}\big|^2+2K\Delta\mathbb{E}\big|x(t_{n-1-M})-Z_{n-1-M}\big|^2
				<\infty.
			\end{split}
		\end{equation*}	
		By the induction reasoning, we know that $\mathbb{E}\big|[x(t_n)-Z_n]-\Delta[\alpha(x(t_n),x(t_{n-M}))-\alpha(Z_n,Z_{n-M})]\big|^2<\infty$. Then for
		for all $n\in \{1,2,...,N\}$ , $N\in\mathbb{N}$, the following results can be acquired
		
		\begin{equation*}	
			\mathbb{E}\big| x(t_n)-Z_n\big|^2<\infty,
		\end{equation*}	
		\begin{equation*}	
			\mathbb{E}\big|\alpha(x(t_n),x(t_{n-M}))-\alpha(Z_n,Z_{n-M})\big|^2<\infty,
		\end{equation*}	
		\begin{equation*}	
			\mathbb{E}\big\|\beta(x(t_n),x(t_{n-M}))-\beta(Z_n,Z_{n-M})\big\|^2<\infty.
		\end{equation*}	
	\end{proof}
		With these two lemmas, we are going to prove the following theorem.

	\begin{thm}\label{thm2.5}
		Let Assumptions \ref{a1}, \ref{a2} hold and $0<2K\Delta<1$. Then for all $n\in \{1,2,...,N\}$, $N\in\mathbb{N}$, there is a constant $C$ independent of n , such that
		\begin{equation*}
			\mathbb{E}|x(t_n)-Z_n|^2\leq C\left(\sum_{i=1}^{n}\mathbb{E}|\mathcal{R}_i| ^2 +\Delta^{-1}\sum_{i=1}^{n}\mathbb{E}\big[|\mathbb{E}(\mathcal{R}_i\mid \mathcal{F}_{t_{i-1}})|^2\big]\right),
		\end{equation*}
		where $\mathcal{R}_i$ is defined by (\ref{Rn}).
	\end{thm}
	
	\begin{proof} 
		By Lemma $\ref{le1}$ and the H\"older inequality, we derive that
		\begin{equation*}
			\begin{split}
				& \mathbb{E}\big|[x(t_n)-Z_n]-\Delta[\alpha(x(t_n),x(t_{n-M}))-\alpha(Z_n,Z_{n-M})]\big|^2\\
				%=& \mathbb{E}[|[x(t_{n-1})-Z_{n-1}]-[\alpha(x(t_{n-1}),x(t_{n-1-M}))-\alpha(Z_{n-1},Z_{n-1-M})]\Delta|^2]\\
				%&- \Delta^2\mathbb{E}[|\alpha(x(t_{n-1}),x(t_{n-1-M}))-\alpha(Z_{n-1},Z_{n-1-M})|^2]\\
				%&+ \Delta\mathbb{E}[\|\beta(x(t_{n-1}),x(t_{n-1-M}))-\beta(Z_{n-1},Z_{n-1-M})\|^2]+ \mathbb{E}[|\mathcal{R}_n|^2]]\\	
				%&+2 \Delta\mathbb{E}[ \langle x(t_{n-1})-Z_{n-1} , \alpha(x(t_{n-1}),x(t_{n-1-M}))-\alpha(Z_{n-1},Z_{n-1-M})  \rangle]\\
				%&+2\mathbb{E}[\langle[x(t_{n-1})-Z_{n-1}]-\Delta[\alpha(x(t_{n-1}),x(t_{n-1-M}))-\alpha(Z_{n-1},Z_{n-1-M})] , \mathcal{R}_n\rangle]\\ 
				%&+2\Delta \mathbb{E}[\langle \alpha(x(t_{n-1}),x(t_{n-1-M}))-\alpha(Z_{n-1},Z_{n-1-M}) ,\mathcal{R}_n  \rangle]\\
				%&+2\mathbb{E}[\langle[\beta(x(t_{n-1}),x(t_{n-1-M}))-\beta(Z_{n-1},Z_{n-1-M})]\Delta W_{n-1} , \mathcal{R}_n\rangle]\\
				\leq&\mathbb{E}\big|[x(t_{n-1})-Z_{n-1}]-\Delta[\alpha(x(t_{n-1}),x(t_{n-1-M}))-\alpha(Z_{n-1},Z_{n-1-M})]\big|^2\\
				&+ \Delta\mathbb{E}\big\|\beta(x(t_{n-1}),x(t_{n-1-M}))-\beta(Z_{n-1},Z_{n-1-M})\big\|^2+ \mathbb{E}|\mathcal{R}_n|^2\\	
				&+2 \Delta\mathbb{E}\big\langle x(t_{n-1})-Z_{n-1} , \alpha(x(t_{n-1}),x(t_{n-1-M}))-\alpha(Z_{n-1},Z_{n-1-M}) \big \rangle\\
				&+2\mathbb{E}\big\langle[x(t_{n-1})-Z_{n-1}]-\Delta[\alpha(x(t_{n-1}),x(t_{n-1-M}))-\alpha(Z_{n-1},Z_{n-1-M})] , \mathcal{R}_n \big \rangle\\ 
				&+2\Delta \mathbb{E}\big\langle \alpha(x(t_{n-1}),x(t_{n-1-M}))-\alpha(Z_{n-1},Z_{n-1-M}) ,\mathcal{R}_n   \big \rangle\\
				&+(q-2)\mathbb{E}\big\|[\beta(x(t_{n-1}),x(t_{n-1-M}))-\beta(Z_{n-1},Z_{n-1-M})]\Delta W_{n-1}\big\|^2\\
				&+\frac{1}{q-2}\mathbb{E}|\mathcal{R}_n|^2\\
				=& 
				\mathbb{E}\big|[x(t_{n-1})-Z_{n-1}]-[\alpha(x(t_{n-1}),x(t_{n-1-M}))-\alpha(Z_{n-1},Z_{n-1-M})]\big|^2\\
				&+2 \Delta\mathbb{E}\big \langle x(t_{n-1})-Z_{n-1} , \alpha(x(t_{n-1}),x(t_{n-1-M}))-\alpha(Z_{n-1},Z_{n-1-M})  \big \rangle\\
				&+(q-1)\Delta\mathbb{E}\big|\beta(x(t_{n-1}),x(t_{n-1-M}))-\beta(Z_{n-1},Z_{n-1-M})\big|^2+\frac{q-1}{q-2}\mathbb{E}|\mathcal{R}_n|^2\\
				&+2\mathbb{E}\big\langle x(t_{n-1})-Z_{n-1} , \mathcal{R}_n \big \rangle.\\ 
			\end{split}
		\end{equation*}
		Using Assumption \ref{a1} and the properties of conditional expectation leads to
		\begin{equation*}
			\begin{split}
				\mathbb{E}&\big|[x(t_n)-Z_n]-\Delta[\alpha(x(t_n),x(t_{n-M}))-\alpha(Z_n,Z_{n-M})]\big|^2\\
				\leq&\mathbb{E}\big|[x(t_{n-1})-Z_{n-1}]-\Delta[\alpha(x(t_{n-1}),x(t_{n-1-M}))-\alpha(Z_{n-1},Z_{n-1-M})]\big|^2\\
				&+2K\Delta\mathbb{E}\big|x(t_{n-1})-Z_{n-1}\big|^2+2K\Delta\mathbb{E}\big|x(t_{n-1-M})-Z_{n-1-M}\big|^2\\
				&+\Delta\mathbb{E}\big|x(t_{n-1})-Z_{n-1}\big|^2+\frac{1}{\Delta}\mathbb{E}\big[|\mathbb{E}(\mathcal{R}_i\mid \mathcal{F}_{t_{i-1}})|^2\big]+\frac{q-1}{q-2}\mathbb{E}|\mathcal{R}_n|^2.
			\end{split}
		\end{equation*}
		By iterating and $\mathbb{E}\big|[x(t_0)-Z_0]-\Delta[\alpha(x(t_0),x(-\tau))-\alpha(Z_0,Z_{-M})]\big|=0$, we have
		\begin{equation*}
			\begin{split}
				&\mathbb{E}\big|[x(t_n)-Z_n]-\Delta[\alpha(x(t_n),x(t_{n-M}))-\alpha(Z_n,Z_{n-M})]\big|^2\\
				\leq&\mathbb{E}\big|[x(t_0)-Z_0]-\Delta[\alpha(x(t_0),x(-\tau))-\alpha(Z_0,Z_{-M})]\big|^2\\
				&+(2K+1)\Delta\sum_{i=0}^{n-1}\mathbb{E}\big|x(t_i)-Z_i\big|^2+2K\Delta\sum_{i=0}^{n-1}\mathbb{E}\big|x(t_{i-M})-Z_{i-M}\big|^2\\
				&+\frac{1}{\Delta}\sum_{i=1}^{n}\mathbb{E}\big[|\mathbb{E}(\mathcal{R}_i\mid \mathcal{F}_{t_{i-1}})|^2\big]+\frac{q-1}{q-2}\sum_{i=1}^{n}\mathbb{E}|\mathcal{R}_i| ^2\\
				\leq&(2K+1)\Delta\sum_{i=0}^{n-1}\mathbb{E}\big|x(t_i)-Z_i\big|^2+2K\Delta\sum_{i=-M}^{n-1-M}\mathbb{E}\big|x(t_i)-Z_i\big|^2\\
				&+\frac{1}{\Delta}\sum_{i=1}^{n}\mathbb{E}\big[|\mathbb{E}(\mathcal{R}_i\mid \mathcal{F}_{t_{i-1}})|^2\big]+\frac{q-1}{q-2}\sum_{i=1}^{n}\mathbb{E}|\mathcal{R}_i| ^2,\\
			\end{split}
		\end{equation*}
		where we used the fact that
		$
			\sum_{i=0}^{n-1}\mathbb{E}\big|x(t_{i-M})-Z_{i-M}\big|^2= \sum_{i=-M}^{n-1-M}\mathbb{E}\big|x(t_i)-Z_i\big|^2.
	$
		Using the Assumption \ref{a1} leads to
		
		\begin{equation*}
			\begin{split}
				&\mathbb{E}\big|[x(t_n)-Z_n]-\Delta[\alpha(x(t_n),x(t_{n-M}))-\alpha(Z_n,Z_{n-M})]\big|^2\\
				=&\mathbb{E}\big|x(t_n)-Z_n\big|^2-2\Delta\mathbb{E}\big\langle x(t_n)-Z_n , \alpha(x(t_n),x(t_{n-M}))-\alpha(Z_n,Z_{n-M})\big\rangle\\
				&+\Delta^2\mathbb{E}\big|\alpha(x(t_n),x(t_{n-M}))-\alpha(Z_n,Z_{n-M})\big|^2\\
				\geq&(1-2K\Delta)\mathbb{E}\big|x(t_n)-Z_n\big|^2-2K\Delta\mathbb{E}\big|x(t_{n-M})-Z_{n-M}\big|^2.
			\end{split}
		\end{equation*}
		Combining these inequalities yields that
		
		\begin{equation*} 
			\begin{split}
				&(1-2K\Delta)\mathbb{E}\big|x(t_n)-Z_n\big|^2\\
				\leq
				&(2K+1)\Delta\sum_{i=0}^{n-1}\mathbb{E}\big|x(t_i)-Z_i\big|^2+2K\Delta\sum_{i=-M}^{-1}\mathbb{E}\big|x(t_i)-Z_i\big|^2\\
				&+2K\Delta\sum_{i=0}^{n-1-M}\mathbb{E}\big|x(t_i)-Z_i\big|^2+2K\Delta\mathbb{E}\big|x(t_{n-M})-Z_{n-M}\big|^2\\
				&+\frac{1}{\Delta} \sum_{i=1}^{n}\mathbb{E}\big[|\mathbb{E}(\mathcal{R}_i\mid \mathcal{F}_{t_{i-1}})|^2\big]+\frac{q-1}{q-2}\sum_{i=1}^{n}\mathbb{E}|\mathcal{R}_i| ^2\\
				\leq
				&(2K+1)\Delta\sum_{i=0}^{n-1}\mathbb{E}\big|x(t_i)-Z_i\big|^2+2K\Delta\sum_{i=0}^{n-M}\mathbb{E}\big|x(t_i)-Z_i\big|^2\\
				&+\frac{1}{\Delta} \sum_{i=1}^{n}\mathbb{E}\big[|\mathbb{E}(\mathcal{R}_i\mid \mathcal{F}_{t_{i-1}})|^2\big]+\frac{q-1}{q-2}\sum_{i=1}^{n}\mathbb{E}|\mathcal{R}_i| ^2\\
				\leq
				&(4K+1)\Delta\sum_{i=0}^{n-1}\mathbb{E}\big|x(t_i)-Z_i\big|^2+\frac{1}{\Delta} \sum_{i=1}^{n}\mathbb{E}\big[|\mathbb{E}(\mathcal{R}_i\mid \mathcal{F}_{t_{i-1}})|^2\big]+\frac{q-1}{q-2}\sum_{i=1}^{n}\mathbb{E}|\mathcal{R}_i| ^2.\\
			\end{split}
		\end{equation*}
		By the Gronwall inequality, we get the desired result.
		
	\end{proof} 
	
	\section{Strong convergence rate}
	The strong convergence rate of BEM for SDDE with super-linear coefficients is dicussed in this section.
	% under globally polynomial growth condition .Using the conclusion thay has been drawn above, we will analyze two kinds of SDDEs with different noise, one SDDE is with multiplicative noise , the other is with additive noise .There is a distinct convergence rate between the two equations.
	In order to analyze the convergence rate, we need to make  additional assumptions.
	\begin{ass} \label{a3}
		Let Assumption \ref{a1} hold, and there are several constants $\rho\in[1,\infty)$, $\bar{p}\in[4\rho-2 , \infty)$, $K_1\in(0,\infty)$ such that
		\begin{equation*} \label{equ4}
			\big|\alpha(x,y)-\alpha(\bar{x},\bar{y})\big|\leq K_{1}	\big(1+|x|^{\rho-1}+|y|^{\rho-1}+|\bar{x}|^{\rho-1}+|\bar{y}|^{\rho-1}	\big)\big(|x-\bar{x}|+|y-\bar{y}|	\big)
		\end{equation*}
		\begin{equation*}\label{key}
			\big\langle x , \alpha(x,y)	\big\rangle+\frac{(\bar{p}-1)}{2}	\big\|\beta(x,y)\big\|^{2}\leq K_{1}\big(1+|x|^{2}+|y|^{2}\big),
		\end{equation*}
		for any $x,y,\bar{x},\bar{y}\in \mathbb{R}^d$.
		
		%Meanwhile , suppose that the initial data $x(t_0)$ is $\mathcal{F}_0-adapted$ and belongs to space $\mathcal{L}^{\bar{p}}$$( \Omega, \mathbb{R}^d )$
		
		%\begin{equation*}
		%\mathbb{E}{|x(t_0)|}^{\bar{p}}<\infty
		%\end{equation*}
	\end{ass}
	By Young's inequality, Assumption \ref{a3} leads to
	\begin{equation} \label{sss}
		\begin{split}
		&\|\beta(x,y)-\beta(\bar{x},\bar{y})\|^2\leq C(1+|x|^{\rho-1}+|y|^{\rho-1}+|\bar{x}|^{\rho-1}+|\bar{y}|^{\rho-1})(|x-\bar{x}|^2+|y-\bar{y}|^2),
		\end{split}
	\end{equation}
\begin{equation*}
	|(\alpha(x,y)|\leq C\left( 1+|x|^{\rho}+|y|^{\rho}\right),
\end{equation*}
	and
	\begin{equation} \label{ssss}
		\|(\beta(x,y)\|^2\leq C\left (1+|x|^{\rho+1}+|y|^{\rho+1} \right),
	\end{equation}
	for any $x,y,\bar{x},\bar{y}\in \mathbb{R}^d$.

	It is worth noting that Assumption \ref{a3} suffices to imply Assumption \ref{a2}, which means that SDDE (\ref{sdde1}) admits a unique solution satisfying $\sup_{s\in[0,T]}\mathbb{E}|x(s)|^2<\infty$, 
	$\sup_{s\in[0 , T]}\mathbb{E}\big|\alpha(x(s),x(s-\tau))\big|^2<\infty$. Moreover, under Assumption \ref{a3}, for any $p\in[2 , \bar{p}]$, 
	the exact solution of SDDE (\ref{sdde1}) with the initial data (\ref{initial}) satisfies
	%It should be noted that the initial data $x_{0}\in \mathcal{C}_{\mathcal{F}_0}^b ([ - \tau , 0 ];\mathbb{R}^d)$, also owing to Assumptions \ref{a1} and \ref{a3}, the exact solutions of SDDE(\ref{sdde1}) for any $p\in[2 , \bar{p}]$ satisfies
	%x_{0}=\varphi=\{ \varphi ( \theta ) : - \tau \leq \theta \leq 0 \}\in \mathcal{C}_{\mathcal{F}_0}^b ([ - \tau , 0 ];\mathbb{R}^d),
	\begin{equation*} 
		\sup_{t\in[0 , T]}\mathbb{E}|x(t)|^p<\infty.
	\end{equation*}
	
	\begin{ass}\label{ass3.2}
		There exists a constant $K_2> 0 $ such that the initial value $\varphi$ satisfies
		\begin{equation*} 
			|\varphi({\theta_2})-\varphi({\theta_1})| \leq K_2|{\theta_2}-{\theta_1}|^\frac{1}{2},~~~~~~~-\tau\leq{\theta_1}<{\theta_2}\leq0.
		\end{equation*}
	\end{ass}

	\begin{lem}
		Let Assumption \ref{a3} hold. For any $0\leq t_1<t_2\leq T$ , we can derive that
		\begin{equation*} 
			\left( \mathbb{E}\left|x(t_2)-x(t_1)\right|^\gamma\right)^\frac{1} {\gamma}\leq C(t_2-t_1)^\frac{1}{2}, ~~~~~\forall\gamma\in[2,\bar{p}/\rho].
		\end{equation*}
	\end{lem}
	
	\begin{proof}
		Using an elementary inequality $|a+b|^\gamma \leq 2^{\gamma-1}(|a|^\gamma+|b|^\gamma)$ for any $a,b\in\mathbb{R}^d$, it is easy to get that
		
		\begin{equation*}
			\begin{split}
				&\mathbb{E}\left|x(t_2)-x(t_1)\right|^\gamma\\
				\leq &2^{\gamma-1}\mathbb{E}\left| \int_{t_1}^{t_2}\alpha\left(x(s), x(s-\tau)\right) ds\right| ^\gamma+2^{\gamma-1}\mathbb{E}\left| \int_{t_1}^{t_2}\beta\left(x(s), x(s-\tau)\right) dW(s)\right| ^\gamma.
			\end{split}
		\end{equation*}
		Then by	Assumption \ref{a3}, H\"older's inequality and Theorem 1.7.1 in \cite{18}, one can derive that 
		
		\begin{equation*}
			\begin{split}
				&\mathbb{E}\left|x(t_2)-x(t_1)\right|^\gamma\\
				\leq& [2(t_2-t_1)]^{\gamma-1}\mathbb{E} \int_{t_1}^{t_2}|\alpha\left(x(s), x(s-\tau)\right)|^\gamma ds\\ &+\frac{1}{2}[2\gamma(\gamma-1)]^\frac{\gamma}{2}(t_2-t_1)^\frac{\gamma-2}{2}\mathbb{E}\int_{t_1}^{t_2}|\beta\left(x(s), x(s-\tau)\right)|^\gamma ds\\
				\leq&C(t_2-t_1)^{\gamma-1}\int_{t_1}^{t_2}\mathbb{E}\left( 1+|x(s)|^{\gamma\rho}+|x(s-\tau)|^{\gamma\rho}\right) ds\\
				&+C(t_2-t_1)^\frac{\gamma-2}{2}\int_{t_1}^{t_2}\mathbb{E}\left (1+|x(s)|^\frac{\gamma(\rho+1)}{2}+|x(s-\tau)|^\frac{\gamma(\rho+1)}{2} \right) ds\\
				\leq&C(t_2-t_1)^\frac{\gamma}{2}.
			\end{split}
		\end{equation*}
	\end{proof}
	\begin{thm} \label{thm3.2}
		Let Assumption \ref{a3} hold and $\Delta\in(0 , \frac{1}{2K})$. Then for the exact solution x(t) to SDDE(\ref{sdde1}) and the numerical solution $Z_n$ to BEM(\ref{bem1})% which denoted by $\{x(t)\}_{t\in[0,T]}$ and $\{Z_n\}_{0\leqslant n \leqslant N}$
		, there is a constant $C$ independent of n, $\Delta$ such that
		\begin{equation*} 
			\sup_{0\leq n \leq N}\mathbb{E}|x(t_n)-Z_n|^2\leq C\Delta.
		\end{equation*}
	\end{thm}
	
	\begin{proof}
		According to Theorem \ref{thm2.5}, to obtain the convergence rate, what we need to do is estimating two terms $\mathbb{E}[|\mathbb{E}(\mathcal{R}_i\mid \mathcal{F}_{t_{i-1}})|^2]$ and $\mathbb{E}|\mathcal{R}_i|^2$, $i\in\{1,2,...,N\}$. 
		By (\ref{Rn}), we have
		\begin{equation*}
			\begin{split}
				\mathbb{E}|\mathcal{R}_i|^2 &= \mathbb{E}\Big| \int_{t_{i-1}}^{t_i} [\alpha(x(s) , x(s-\tau))-\alpha(x(t_i)-x(t_{i-M}))]ds\\
				&+ \int_{t_{i-1}}^{t_i} [\beta(x(s) , x(s-\tau))-\beta(x(t_{i-1})-x(t_{i-1-M}))]dW(s)\Big| ^2\\
				\leq&2\mathbb{E}\Big| \int_{t_{i-1}}^{t_i} [\alpha(x(s) , x(s-\tau))-\alpha(x(t_i)-x(t_{i-M}))]ds\Big|^2\\
				&+ 2\mathbb{E}\Big| \int_{t_{i-1}}^{t_i} [\beta(x(s) , x(s-\tau))-\beta(x(t_{i-1})-x(t_{i-1-M}))]dW(s)\Big| ^2.\\
			\end{split}
		\end{equation*}
		Now we estimate the first term on the right side of the inequality. By the H\"older inequality and Assumptions \ref{a3}, \ref{ass3.2}, we can get
		\begin{equation*}
			\begin{split}
				&\mathbb{E}\left| \int_{t_{i-1}}^{t_i} [\alpha(x(s) , x(s-\tau))-\alpha(x(t_i)-x(t_{i-M}))]ds\right| ^2\\
				\leq& C\Delta\mathbb{E}\int_{t_{i-1}}^{t_i} \left| \alpha(x(s) , x(s-\tau))-\alpha(x(t_i)-x(t_{i-M}))\right|^2 ds\\ 
				\leq&C\Delta\int_{t_{i-1}}^{t_i} \mathbb{E} \Big[ (1+|x(s)|^{2\rho-2}+|x(s-\tau)|^{2\rho-2}+|x(t_i)|^{2\rho-2}+|x(t_{i-M})|^{2\rho-2})\\
				&~~~~~~~~~~~~~~\cdot(|x(s)-x(t_i)|^2+|x(s-\tau)-x(t_{i-M})|^2)\Big]ds \\
				\leq&C\Delta\int_{t_{i-1}}^{t_i} \big(1+\mathbb{E}|x(s)|^{\frac{\bar{p}(2\rho-2)}{\bar{p}-2\rho}}+\mathbb{E}|x(s-\tau)|^{\frac{\bar{p}(2\rho-2)}{\bar{p}-2\rho}}\\
				&~~~~~~~~~~~~~~~+\mathbb{E}|x(t_i)|^{\frac{\bar{p}(2\rho-2)}{\bar{p}-2\rho}}+\mathbb{E}|x(t_{i-M})|^{\frac{\bar{p}(2\rho-2)}{\bar{p}-2\rho}}\big)^\frac{\bar{p}-2\rho}{\bar{p}}\\
				&~~~~~~~~~~~~~\cdot\Big[\left( \mathbb{E}|x(s)-x(t_i)|^\frac{\bar{p}}{\rho}\right) ^\frac{2\rho}{\bar{p}}+\big(\mathbb{E}|x(s-\tau)-x(t_{i-M})|^\frac{\bar{p}}{\rho}\big) ^\frac{2\rho}{\bar{p}}\Big]ds\\
				\leq&C\Delta^3.
			\end{split}
		\end{equation*}
		Employing the It\^{o} isometry gives that
		\begin{equation*}
			\begin{split}
				&\mathbb{E}\left| \int_{t_{i-1}}^{t_i} [\beta(x(s) , x(s-\tau))-\beta(x(t_{i-1})-x(t_{i-1-M}))]dW(s)\right|^2\\
				=&\mathbb{E}\int_{t_{i-1}}^{t_i} \big\|\beta(x(s) , x(s-\tau))-\beta(x(t_{i-1})-x(t_{i-1-M}))\big\|^2ds\\
				\leq&C\int_{t_{i-1}}^{t_i} \mathbb{E} \Big[ (1+|x(s)|^{\rho-1}+|x(s-\tau)|^{\rho-1}+|x(t_i)|^{\rho-1}+|x(t_{i-M})|^{\rho-1})\\
				&~~~~~~~~~~~~~\cdot(|x(s)-x(t_i)|^2+|x(s-\tau)-x(t_{i-M})|^2)\Big]ds \\
				\leq&C\Delta^2.
			\end{split}
		\end{equation*}
		Thus, by Theorem \ref{thm2.5}, we draw a conclusion that
		\begin{equation} \label{ER}
			\mathbb{E}|\mathcal{R}_i|^2\leq K_{1}\Delta^2.
		\end{equation}
		Moreover,
		\begin{equation} \label{ER2}
			\begin{split}
				&\mathbb{E}\big[|\mathbb{E}(\mathcal{R}_i\mid \mathcal{F}_{t_{i-1}})|^2\big]\\
				&=\mathbb{E}\Big[\big|\mathbb{E}\big( \int_{t_{i-1}}^{t_i} [\alpha(x(s) , x(s-\tau))-\alpha(x(t_i)-x(t_{i-M}))]ds\mid \mathcal{F}_{t_{i-1}}\big)\big|^2\Big]\\
				&\leq \mathbb{E}\big|\int_{t_{i-1}}^{t_i} [\alpha(x(s) , x(s-\tau))-\alpha(x(t_i)-x(t_{i-M}))]ds\big |^2\\
				&\leq C\Delta^3,
			\end{split}
		\end{equation}
		where we used the fact that
		\begin{equation*}
		 \mathbb{E}\left( \int_{t_{i-1}}^{t_i} [\beta(x(s) , x(s-\tau))-\beta(x(t_{i-1})-x(t_{i-1-M}))]dW(s)\mid \mathcal{F}_{t_{i-1}}\right) =0.
		\end{equation*}
		Combining (\ref{ER}) and (\ref{ER2}) gives the result.
	\end{proof}
	\section{Convergence rate under the stronger condition}
	This section shows the convergence rate of BEM (\ref{bem1}) for SDDE (\ref{sdde1}) under stronger condition by using a simpler proof process.
	\begin{ass}\label{4.1}
		There exist constants $K_3,K_4>0$ such that
		\begin{equation*}
			\begin{split}
				&\langle x-\bar{x} , \alpha(x,y)-\alpha(\bar{x},\bar{y})\rangle +\|\beta(x,y)-\beta(\bar{x},\bar{y})\|^{2}\\
				\leq& -K_3|x-\bar{x}|^{2}+K_4|y-\bar{y}|^{2},
			\end{split}
		\end{equation*}
		for any $x,y,\bar{x},\bar{y}\in \mathbb{R}^d$.
	\end{ass}
	%\begin{ass}\label{4.2}
	%There are several constants $K_5, K_6, K_7>0$ such that
	%\begin{equation*}
	%	\langle x , \alpha(x,y)\rangle+\|\beta(x,y)\|^{2}\leq -%K_5|x|^{2}+K_6|y|^{2}+K_7
	%\end{equation*}
	%for any $t\in [0,T]$ and $x,y,\bar{x},\bar{y}\in \mathbb{R}^d$.
	%\end{ass}
	 Under Assumptions \ref{a3} and \ref{4.1}, (\ref{sss}) and (\ref{ssss}) can be obtained similarly.
	\begin{thm} \label{thm4.1}
		Let Assumptions \ref{a3}, \ref{4.1} hold and $\Delta\in(0 , \frac{1}{K})$.Then for the exact solution x(t) to SDDE (\ref{sdde1}) and the numerical solution $Z_n$ to BEM (\ref{bem1})% which denoted by $\{x(t)\}_{t\in[0,T]}$ and $\{Z_n\}_{0\leqslant n \leqslant N}$
		, there is a constant $C$ independent of n, $\Delta$ satisfying
		\begin{equation*} 
			\sup_{0\leq n \leq N}\mathbb{E}|x(t_n)-Z_n|^2\leq C\Delta.
		\end{equation*}
	\end{thm}
	\begin{proof}
		According to the definitions of (\ref{sdde1}) and (\ref{bem1}), we have 
		\begin{equation*}
			\begin{split}
				&x(t_n)-Z_n\\
				=&[x(t_{n-1})-Z_{n-1}]+\Delta[\alpha(x(t_n),x(t_{n-M}))-\alpha(Z_n,Z_{n-M})]\\
				&+[\beta(x(t_{n-1}),x(t_{n-1-M}))-\beta(Z_{n-1},Z_{n-1-M})]\Delta W_{n-1}\\
				&+\mathcal{R}_n ,
			\end{split}
		\end{equation*}
		where $\mathcal{R}_n$ is defined by (\ref{Rn}). 
		Then
		\begin{equation*}
			\begin{split}
				&2\mathbb{E}\big\langle[x(t_n)-Z_n]-[x(t_{n-1})-Z_{n-1}],x(t_n)-Z_n\big\rangle\\
				=&2\Delta\mathbb{E}\big\langle\alpha(x(t_n),x(t_{n-M}))-\alpha(Z_n,Z_{n-M}),x(t_n)-Z_n\big\rangle\\
				&+2\mathbb{E}\big\langle[\beta(x(t_{n-1}),x(t_{n-1-M}))-\beta(Z_{n-1},Z_{n-1-M})]\Delta W_{n-1},x(t_n)-Z_n\big\rangle\\
				&+2\mathbb{E}\big\langle\mathcal{R}_n,x(t_n)-Z_n\big\rangle\\
				=&2\Delta\mathbb{E}\big\langle\alpha(x(t_n),x(t_{n-M}))-\alpha(Z_n,Z_{n-M}),x(t_n)-Z_n\big\rangle\\
				&+2\mathbb{E}\big\langle[\beta(x(t_{n-1}),x(t_{n-1-M}))-\beta(Z_{n-1},Z_{n-1-M})]\Delta W_{n-1},\\
				&~~~~~[x(t_n)-Z_n]-[x(t_{n-1})-Z_{n-1}]\big\rangle\\
				&+2\mathbb{E}\big\langle\mathbb{E}\left( \mathcal{R}_n\mid \mathcal{F}_{t_{i-1}}\right),x(t_{n-1})-Z_{n-1}\big\rangle\\
				&+2\mathbb{E}\big\langle\mathcal{R}_n,[x(t_n)-Z_n]-[x(t_{n-1})-Z_{n-1}]\big\rangle.
			\end{split}
		\end{equation*} 
		By Young's inequality and Assumption \ref{4.1}, there exists a positive constant $\varepsilon_*$ such that
		\begin{equation*}
			\begin{split}
				&\mathbb{E}\big|x(t_n)-Z_n\big|^2-\mathbb{E}\big|x(t_{n-1})-Z_{n-1}\big|^2+\mathbb{E}\big|[x(t_n)-Z_n]-[x(t_{n-1})-Z_{n-1}]\big|^2\\
				\leq&-2K_3\Delta\mathbb{E}\big|x(t_n)-Z_n\big|^2+2K_4\Delta\mathbb{E}\big|x(t_{n-M})-Z_{n-M}\big|^2\\
				&-2\Delta\mathbb{E}\big\|\beta(x(t_n),x(t_{n-M}))-\beta(Z_n,Z_{n-M})\big\|^2\\
				%\big\langle\alpha(x(t_n),x(t_{n-M}))-\alpha(Z_n,Z_{n-M}),x(t_n)-Z_n\big\rangle\\
				&+2\Delta\mathbb{E}\big\|\beta(x(t_{n-1}),x(t_{n-1-M}))-\beta(Z_{n-1},Z_{n-1-M})\big\|^2\\
				&+2\mathbb{E}|\mathcal{R}_n|^2+\mathbb{E}\big|[x(t_n)-Z_n]-[x(t_{n-1})-Z_{n-1}]\big|^2\\
				&+\frac{1}{\varepsilon_*^2}\mathbb{E}\big[|\mathbb{E}(\mathcal{R}_i\mid \mathcal{F}_{t_{i-1}})|^2\big]+\varepsilon_*^2\mathbb{E}\big |x(t_{n-1})-Z_{n-1}\big|^2,
			\end{split}
		\end{equation*} 
		where we used the equliaty (\ref{cru}).
		Then the inequality can be rearranged as
		\begin{equation}\label{sec4inequ}
			\begin{split}
				&\big(1+2K_3\Delta\big)\mathbb{E}\big |x(t_n)-Z_n\big|^2+2\Delta\mathbb{E}\big\|\beta(x(t_n),x(t_{n-M}))-\beta(Z_n,Z_{n-M})\big\|^2\\
				\leq&\big(1+\varepsilon_*^2\big)\mathbb{E}\big|x(t_{n-1})-Z_{n-1}\big|^2\\
				&+2\Delta\mathbb{E}\big\|\beta(x(t_{n-1}),x(t_{n-1-M}))-\beta(Z_{n-1},Z_{n-1-M})\big\|^2\\
				&+2K_4\Delta\mathbb{E}\big|x(t_{n-M})-Z_{n-M}\big|^2+2\mathbb{E}|\mathcal{R}_n|^2+\frac{1}{\varepsilon_*^2}\mathbb{E}\big[|\mathbb{E}(\mathcal{R}_i\mid \mathcal{F}_{t_{i-1}})|^2\big].
			\end{split}
		\end{equation}
		By choosing $\varepsilon_*=\sqrt{2K_3\Delta}$ and denoting ${G_n}=\big(1+2K_3\Delta\big)\mathbb{E}\big |x(t_n)-Z_n\big|^2+2\Delta\mathbb{E}\big\|\beta(x(t_n),x(t_{n-M}))-\beta(Z_n,Z_{n-M})\big\|^2$, we get that
		\begin{equation*}
			{G_n}-{G_{n-1}}\leq2K_4\Delta\mathbb{E}\big|x(t_{n-M})-Z_{n-M}\big|^2+2\mathbb{E}|\mathcal{R}_n|^2+\frac{1}{2K_3\Delta}\mathbb{E}\big[|\mathbb{E}(\mathcal{R}_i\mid \mathcal{F}_{t_{i-1}})|^2\big].
		\end{equation*}
		It is easy to see that
		\begin{equation*}
			\begin{split}
				{G_n}\leq&2K_4\Delta\sum_{i=1}^{n}\mathbb{E}\big|x(t_{i-M})-Z_{i-M}\big|^2+2\sum_{i=1}^{n}\mathbb{E}|\mathcal{R}_n|^2\\
				&+\frac{1}{2K_3\Delta}\sum_{i=1}^{n}\mathbb{E}\big[|\mathbb{E}(\mathcal{R}_i\mid \mathcal{F}_{t_{i-1}})|^2\big]\\
				\leq&2K_4\Delta\sum_{i=1-M}^{0}\mathbb{E}\big|x(t_i)-Z_i\big|^2+2K_4\Delta\sum_{i=1}^{n-1}\mathbb{E}\big|x(t_i)-Z_i\big|^2\\
				&+2\sum_{i=1}^{n}\mathbb{E}|\mathcal{R}_n|^2+\frac{1}{2K_3\Delta}\sum_{i=1}^{n}\mathbb{E}\big[|\mathbb{E}(\mathcal{R}_i\mid \mathcal{F}_{t_{i-1}})|^2\big]\\
				\leq&2K_4\Delta\sum_{i=1}^{n-1}\mathbb{E}\big|x(t_i)-Z_i\big|^2+2\sum_{i=1}^{n}\mathbb{E}|\mathcal{R}_n|^2+\frac{1}{2K_3\Delta}\sum_{i=1}^{n}\mathbb{E}\big[|\mathbb{E}(\mathcal{R}_i\mid \mathcal{F}_{t_{i-1}})|^2\big],
			\end{split}
		\end{equation*}
		which implies that
		\begin{equation*}
		\begin{split}
			&\mathbb{E}\big |x(t_n)-Z_n\big|^2\\
			&\leq2K_4\Delta\sum_{i=1}^{n-1}\mathbb{E}\big|x(t_i)-Z_i\big|^2+2\sum_{i=1}^{n}\mathbb{E}|\mathcal{R}_n|^2+\frac{1}{2K_3\Delta}\sum_{i=1}^{n}\mathbb{E}\big[|\mathbb{E}(\mathcal{R}_i\mid \mathcal{F}_{t_{i-1}})|^2\big].
		\end{split}
		\end{equation*}
		Using the discrete-type Gronwall inequality, we can obtain the convergence rate.
	\end{proof}
	
	\begin{rem}
		The reason why the technique in Theorem \ref{thm4.1} can simplify the proof process under the stronger condition is that: in (\ref{sec4inequ}), when $1+2K_3\Delta>0$ holds, the subsequent proof process can be given. If Assumption \ref{a1} holds but Assumption \ref{4.1} does not, $1+2K_3\Delta$ will change into $1-2K_3\Delta$, then we can not use this technique to get the desired result.
	\end{rem}
	
	\section{Mean-square stability of BEM}
	
	This section will show that the BEM can inherit the exponential mean-square stability under the fairly general conditions.
	\begin{ass}\label{ass5.1}
	There exist some constants $l>1$, $\Gamma>2$, $c_1>c_2>0, c_3>c_4>0$ such that
	\begin{equation*} 
		\big\langle x , \alpha(x,y)\big\rangle+\frac{l}{2}\big\|\beta(x,y)\big\|^{2}\leq {-c_1}|x|^2+{c_2}|y|^2-{c_3}|x|^{\Gamma}+{c_4}|y|^{\Gamma},
	\end{equation*}
	for all $x,y\in \mathbb{R}^d$.
\end{ass}
	
\begin{defn}
	The exact solution of (\ref{sdde1}) is said to be exponentially mean-square stable if there exists a constant $\varepsilon>0$ such that
		\begin{equation*}
		\limsup_{t\rightarrow\infty}\frac{\log\mathbb{E}|x(t)|^{2}}{t}\le-\varepsilon.
	\end{equation*}
\end{defn}

\begin{defn}
The numerical solution defined by (\ref{bem1}) is said to be exponentially mean-square stable if there exists a constant $\varepsilon>0$ such that, for any $\Delta< \frac{1}{K}$,
\begin{equation*}
\limsup_{n\rightarrow\infty}\frac{\log\mathbb{E}|Z_n|^{2}}{n\Delta}\le-\varepsilon.
		\end{equation*}
\end{defn}

	\begin{thm}
		Under Assumption \ref{ass5.1}, the exact solution of  (\ref{sdde1}) is exponentially mean-square stable.
		\end{thm}
	The proof of above theorem is the same as Theorem 3.1 in \cite{35}, so we omit it. Then we simply prove that the numerical solution to BEM is exponentially mean-square stable by using the novel technique.
	\begin{thm}\label{thm5.2}
		Let Assumption \ref{ass5.1} and (\ref{*}) hold. Then there exists a sufficiently small $\Delta<\min\{\frac{l-1}{2c_1},\Delta^*, \frac{1}{K}\}$ such that
		\begin{equation*}
			\limsup_{n\rightarrow\infty}\frac{\log\mathbb{E}|Z_n|^{2}}{n\Delta}\le-\varepsilon,
		\end{equation*}
		where $\varepsilon$ satisfies $\varepsilon+2{c_2}e^{\varepsilon\tau}<2{c_1}$, and $\Delta^*$ is the root of  $1+2{c_1}\Delta-e^{\varepsilon \Delta}-2{c_2}\Delta e^{\varepsilon \tau}=0$.
	\end{thm}
	\begin{proof}
		Obviously, we can get from (\ref{cru}) and (\ref{bem1}) that
		\begin{equation*}
			\begin{split}
				&2\mathbb{E}\big\langle Z_n-Z_{n-1},Z_n\big\rangle\\
				=&2\Delta\mathbb{E}\big\langle\alpha(Z_n,Z_{n-M}),Z_n\big\rangle+2\mathbb{E}\big\langle \beta(Z_{n-1},Z_{n-1-M})\Delta W_{n-1},Z_n\big\rangle\\
				=&2\Delta\mathbb{E}\big\langle\alpha(Z_n,Z_{n-M}),Z_n\big\rangle+2\mathbb{E}\big\langle \beta(Z_{n-1},Z_{n-1-M})\Delta W_{n-1},Z_n-Z_{n-1}\big\rangle\\
				\leq&-2{c_1}\Delta\mathbb{E}|Z_n|^2+2{c_2}\Delta\mathbb{E}|Z_{n-M}|^2-2{c_3}\Delta\mathbb{E}|Z_n|^{\Gamma}+2{c_4}\Delta\mathbb{E}|Z_{n-M}|^{\Gamma}\\
				&-l\Delta\mathbb{E}\big\|\beta(Z_n,Z_{n-M})\big\|^{2}+\Delta\mathbb{E}\big\|\beta(Z_{n-1},Z_{n-1-M})\big\|^{2}+\mathbb{E}\big|Z_n-Z_{n-1}\big|^2.\\
			\end{split}
		\end{equation*}
		Rearranging this inequality, we have
		\begin{equation*}
			\begin{split}
				&\big(1+2{c_1}\Delta\big)\mathbb{E}|Z_n|^{2}+l\Delta\mathbb{E}\big\|\beta(Z_n,Z_{n-M})\big\|^{2}\\
				\leq&\mathbb{E}|Z_{n-1}|^{2}+\Delta\mathbb{E}\big\|\beta(Z_{n-1},Z_{n-1-M})\big\|^{2}+2{c_2}\Delta\mathbb{E}|Z_{n-M}|^2\\
				&-2{c_3}\Delta\mathbb{E}|Z_n|^{\Gamma}+2{c_4}\Delta\mathbb{E}|Z_{n-M}|^{\Gamma}.\\
			\end{split}
		\end{equation*}
	Due to $\Delta<\frac{l-1}{2c_1}$, it is easy to get
		\begin{equation*}
		\begin{split}
			\big(1+2{c_1}\Delta\big)F_n
			\leq F_{n-1}+2{c_2}\Delta\mathbb{E}|Z_{n-M}|^2
			-2{c_3}\Delta\mathbb{E}|Z_n|^{\Gamma}+2{c_4}\Delta\mathbb{E}|Z_{n-M}|^{\Gamma},\\
		\end{split}
	\end{equation*}
	where $F_n=\mathbb{E}|Z_n|^{2}+\Delta\mathbb{E}\big\|\beta(Z_n,Z_{n-M})\big\|^{2}$.
	By multiplying both sides by $e^{\varepsilon n\Delta}$ and subtracting $\big(1+2{c_1}\Delta\big)e^{\varepsilon (n-1)\Delta}F_{n-1}$ from two sides, we can obtain that
		\begin{equation*}
			\begin{split}
				&\big(1+2{c_1}\Delta\big)e^{\varepsilon n\Delta}F_n-\big(1+2{c_1}\Delta\big)e^{\varepsilon (n-1)\Delta}F_{n-1}\\
				\leq &\big(e^{\varepsilon \Delta}-(1+2{c_1}\Delta)\big)e^{\varepsilon (n-1)\Delta}F_{n-1}+2{c_2}\Delta e^{\varepsilon n\Delta}\mathbb{E}|Z_{n-M}|^2\\
				&-2{c_3}\Delta e^{\varepsilon n\Delta}\mathbb{E}|Z_n|^{\Gamma}+2{c_4}\Delta e^{\varepsilon n\Delta}\mathbb{E}|Z_{n-M}|^{\Gamma}.
			\end{split}
		\end{equation*}
		Then it is not difficult to get that
		\begin{equation*}\label{5.0}
			\begin{split}
				&e^{\varepsilon n\Delta}\mathbb{E}|Z_n|^{2}\leq e^{\varepsilon n\Delta}F_n\\
				\leq&\big(1+2{c_1}\Delta\big)F_0+\big(e^{\varepsilon \Delta}-(1+2{c_1}\Delta)\big)\sum_{i=0}^{n-1}e^{\varepsilon i\Delta}F_i+2{c_2}\Delta \sum_{i=1}^{n}e^{\varepsilon i\Delta}\mathbb{E}|Z_{i-M}|^2\\
				&-2{c_3}\Delta \sum_{i=1}^{n}e^{\varepsilon i\Delta}\mathbb{E}|Z_i|^{\Gamma}+2{c_4}\Delta\sum_{i=1}^{n} e^{\varepsilon i\Delta}\mathbb{E}|Z_{i-M}|^{\Gamma}\\
				\leq&\big(1+2{c_1}\Delta\big)\Big(\mathbb{E}|Z_0|^{2}+\Delta\mathbb{E}\big\|\beta(Z_0,Z_{-M})\big\|^{2}\Big)\\
				&+\big(e^{\varepsilon \Delta}-(1+2{c_1}\Delta)\big)\sum_{i=0}^{n-1}e^{\varepsilon i\Delta}\Big(\mathbb{E}|Z_i|^{2}+\Delta\mathbb{E}\big\|\beta(Z_i,Z_{i-M})\big\|^{2}\Big)\\
				&+2{c_2}\Delta \sum_{i=1-M}^{n-M}e^{\varepsilon (i+M)\Delta}\mathbb{E}|Z_{i}|^2-2{c_3}\Delta \sum_{i=1}^{n}e^{\varepsilon i\Delta}\mathbb{E}|Z_i|^{\Gamma}+2{c_4}\Delta\sum_{i=1-M}^{n-M} e^{\varepsilon (i+M)\Delta}\mathbb{E}|Z_i|^{\Gamma}\\
				\leq&e^{\varepsilon \Delta}\Big(\mathbb{E}|Z_0|^{2}+\Delta\mathbb{E}\big\|\beta(Z_0,Z_{-M})\big\|^{2}\Big)+2{c_2}\Delta e^{\varepsilon \tau}\mathbb{E}|Z_0|^{2}+2{c_4}\Delta e^{\varepsilon \tau}\mathbb{E}|Z_0|^{\Gamma}\\
				&+2{c_2}\Delta e^{\varepsilon \tau}\sum_{i=1-M}^{-1}e^{\varepsilon i\Delta}\mathbb{E}|Z_i|^{2}+2{c_4}\Delta e^{\varepsilon \tau}\sum_{i=1-M}^{-1}e^{\varepsilon i\Delta}\mathbb{E}|Z_i|^{\Gamma}\\
				&-\big(1+2{c_1}\Delta-e^{\varepsilon \Delta}\big)\sum_{i=1}^{n-1}\Delta e^{\varepsilon i\Delta}\mathbb{E}\big\|\beta(Z_i,Z_{i-M})\big\|^{2}\\
				&-\big(1+2{c_1}\Delta-e^{\varepsilon \Delta}-2{c_2}\Delta e^{\varepsilon \tau}\big)\sum_{i=1}^{n-1}\mathbb{E} e^{\varepsilon i\Delta}|Z_i|^{2}-\big(2{c_3}\Delta-2{c_4}\Delta e^{\varepsilon \tau}\big)\sum_{i=1}^{n-1} e^{\varepsilon i\Delta}\mathbb{E}|Z_i|^{\Gamma}.\\
			\end{split}
		\end{equation*}
		Define
		\begin{equation*}
			f(\Delta)=1+2{c_1}\Delta-e^{\varepsilon \Delta}-2{c_2}\Delta e^{\varepsilon \tau}~~~\text{and}~~~
			g(\Delta)=2{c_3}\Delta-2{c_4}\Delta e^{\varepsilon \tau}.
		\end{equation*}
		Then one can see that
		\begin{equation*}
			f'(\Delta)=2{c_1}-\varepsilon e^{\varepsilon\Delta}-2{c_2}e^{\varepsilon\tau},~~~	f''(\Delta)=-\varepsilon^2 e^{\varepsilon\Delta},
		\end{equation*}
	and
		\begin{equation*}
			g'(\Delta)=2{c_3}-2{c_4}e^{\varepsilon\tau}.
		\end{equation*}
		We can observe that $f(0)=0$ and $f'(0)>0$ when $\varepsilon+2{c_2}e^{\varepsilon\tau}<2{c_1}$. It means that $f(\Delta)$ is an increasing function in a sufficiently small interval. In addition, $f''(\Delta)<0$ implies $f(\Delta)$ is concave function, so there exists a $\Delta'$ such that $f'(\Delta')=0$. Then  $f(\Delta)$ increases strictly when $\Delta<\Delta'$ and $f(\Delta)$ decreases strictly when $\Delta>\Delta'$. Therefore, there exists a $\Delta^*>\Delta'$ satisfying $f(\Delta^*)=0$, and for all $\Delta<\Delta^*$
 		\begin{equation*}\label{5.1}
			1+2{c_1}\Delta-e^{\varepsilon \Delta}-2{c_2}\Delta e^{\varepsilon \tau}>0.
		\end{equation*}
		Next we use the same skill to analyze $g(\Delta)$. We can see that $g(0)=0$ and $g'(\Delta)>0$ when ${c_3}>{c_4}$, $\varepsilon\leq\frac{1}{\tau}\log(\frac{c_3}{c_4})$, so $g(\Delta)>0$.

		To sum up, we obtain that for all sufficiently small $\Delta<\min\{\frac{l-1}{2c_1},\Delta^*\,\frac{1}{K}\}$, there exists a positive constant $\Upsilon$ independent of n such that
		\begin{equation*}
			e^{\varepsilon n\Delta}\mathbb{E}|Z_n|^{2}\leq\Upsilon,
		\end{equation*}
		which means that
		\begin{equation*}
			\limsup_{n\rightarrow\infty}\frac{\log\mathbb{E}|Z_n|^{2}}{n\Delta}\le-\varepsilon.
		\end{equation*}
	\end{proof}
\begin{rem}
	We should note that if the corresponding assumptions in \cite{35,28,29} are replaced by Assumption \ref{ass5.1}, the stability results  can not be obtained. But, on the contrary, we can also get Theorem \ref{thm5.2} under the corresponding assumptions in \cite{35,28,29}.
\end{rem}
\begin{rem}
		The locally Lipschitz conditions of drift and diffusion coefficients were used in the proof process in \cite{122}. 
		The advantage of our proof is that, 
		by borrowing the technique in \cite{31}, we can still get the exponential mean-square stability of BEM without using the locally Lipschitz conditions. 
\end{rem}

	\section{Numerical experiments}
	%There are two numerical examples to make our theory realized. Example 1 is used to illustrate the convergence rate of BEM for SDDE with super-linear growing coefficients. Example 2 demostrates the stability of BEM for SDDE with coefficients which satisfying one-sided polynomial growth condition.
	\paragraph{Example 1 }
	Consider
	the following scalar nonlinear SDDE
	\begin{equation}\label{exm}
		d x(t)=\left(x(t)-4x(t)^3+x(t-\tau)\right) d t+\left(x(t)^2-x(t-\tau)+2\right) dW(t)
	\end{equation}
	on $t\in [0,1]$. Here, the initial data $ \varphi (t)=|t|^{\frac{1}{2}}+1$, $t\in[-\tau,0]$.
	Now we verify that the drift and diffusion coeficients fulfill  Assumption \ref{a1}. Let q=3.
	Then,
	
	\begin{equation*}
		\begin{split}
			&\langle x-\bar{x} , \alpha(x,y)-\alpha(\bar{x},\bar{y})\rangle +\frac{(q-1)}{2}|\beta(x,y)-\beta(\bar{x},\bar{y})|^{2}\\
			\leq&|x-\bar{x}|^2-4(x-\bar{x})(x^3-\bar{x}^3)+(x-\bar{x})(y-\bar{y})+2|x^2-\bar{x}^2|^2+2|y-\bar{y}|^2\\
			\leq&|x-\bar{x}|^2+\frac{1}{2}|x-\bar{x}|^2+\frac{1}{2}|y-\bar{y}|^2+2|y-\bar{y}|^2\\
			&-4|x-\bar{x}|^2(x^2+x\bar{x}+\bar{x}^2)+2|x-\bar{x}|^2(x^2+2x\bar{x}+\bar{x}^2)\\
			\leq&|x-\bar{x}|^2+\frac{1}{2}|x-\bar{x}|^2+\frac{1}{2}|y-\bar{y}|^2+2|y-\bar{y}|^2-2|x-\bar{x}|^2(x^2+\bar{x}^2)\\
			\leq&\frac{3}{2}|x-\bar{x}|^2+\frac{5}{2}|y-\bar{y}|^2.
		\end{split}
	\end{equation*}
	And Assumption \ref{a3} is simple to be tested as well. 
	\begin{figure}[htbp] 
		\centering
		\includegraphics[height=7.4cm,width=9cm]{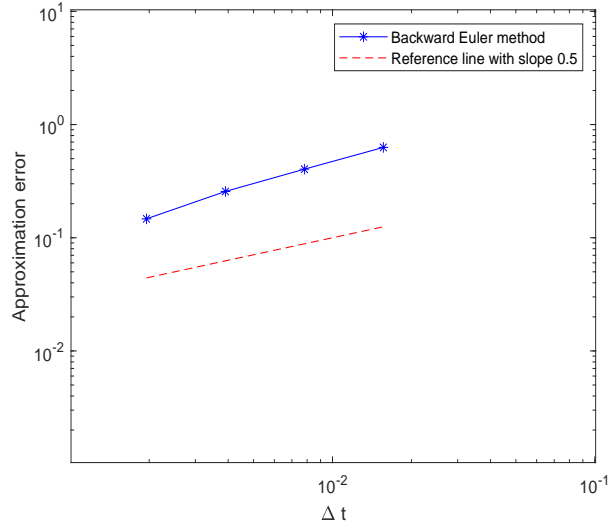}
		\caption{\label{tu1} Convergence rate of BEM for (\ref{exm})}
	\end{figure}
	In order to check the theory in Theorem \ref{thm3.2},  we perform a numerical experiment with four different stepsizes $\Delta=2^{-9}$, $2^{-8}$, $2^{-7}$, $2^{-6}$ at $T=1$. The numerical solution with stepsize $\Delta=2^{-11}$ is regarded as the exact solution of this experiment since it is difficult to be expressed explicitly. Then mean-square error can be estimated by computing the average of 500 sample paths' errors between exact solutions and numerical solutions.
	Figure \ref{tu1} illustrates the mean-square error which is defined by
	\begin{equation*}
		(\mathbb{E}|x(T)-Z_N|^{2})^{\frac{1}{2}}\approx \left(\frac{1}{500}\sum_{i=1}^{500}|x^{i}(T)-Z^{i}_N|^{2}\right)^{\frac{1}{2}}.
	\end{equation*}

	From Figure \ref{tu1}, we can observe that the convergence rate of BEM is close to 0.5, which means that it is coincident with the theoretical conclusion in Theorem \ref{thm3.2}.
	\paragraph{Example 2 }
	Consider
	the following scalar nonlinear SDDE
	\begin{equation}\label{exm2}
		d x(t)=\left(-2x(t)-4x(t)^3+x(t-\tau)\right) d t+\left(x(t)^2+\frac{1}{2}x(t-\tau)\right) dW(t),
	\end{equation}
		on $t\geq 0$. Here, the initial data $ \varphi (t)=|t|^{\frac{1}{2}}+3$, $t\in[-\tau,0]$.	
	Let $l=2$, then we can see that
	\begin{equation*}
		\langle x, \alpha(x,y)\rangle 
		=x(-2x-4x^3+y)
		\leq-\frac{3}{2}|x|^2-4|x|^4+\frac{1}{2}|y|^2,
	\end{equation*}
	\begin{equation*}
		|\beta(x,y)|^{2}=|x^2+\frac{1}{2}y|^{2}\leq2|x|^4+\frac{1}{2}|y|^2,
	\end{equation*}
	which means that the drift and diffusion coefficients satisfy Assumption \ref{ass5.1}.

	\begin{figure}[htbp] 
		\centering
		\includegraphics[height=7.4cm,width=9cm]{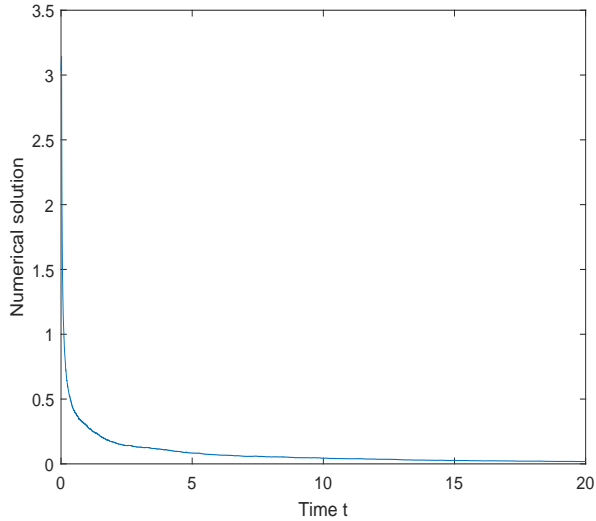}
		\caption{\label{tu2} Mean-square stability of BEM with $\Delta=0.001$}
	\end{figure}

We compute the average of the numerical solutions simulated by 500 sample paths with stepsize $\Delta=0.001$ and plot its trajectory in Figure \ref{tu2}. From Figure \ref{tu2}, we can know that the trajectory tends to  zero as time goes on, which means that the numerical solution of BEM is mean-square stable.

\section*{Acknowledgements}
The authors would like to thank the reviewers for their work.

\section*{Funding}
This work is supported by the National Natural
Science Foundation of China (11871343) and Shanghai Rising-Star Program (22QA1406900).

\section*{Availability of data and materials}
Not applicable.

\section*{Competing interests}
The authors declare that they have no competing interests.

\end{document}